\documentclass[12pt]{amsart}

\usepackage{amsmath,amssymb,amsthm}
\usepackage{amsfonts}
\usepackage[mathscr]{eucal}
\newtheorem{theorem}{Theorem}[section]

\newtheorem{prop}[theorem]{Proposition}

\newtheorem{corollary}[theorem]{Corollary}

\theoremstyle{definition}

\theoremstyle{remark}
\newtheorem{remark}[theorem]{Remark}

\numberwithin{equation}{section}

\newcommand{\RR}[1]{\mathbb{#1}}

\newcommand{\rr}{{\mathbb R}}

\def\R{{\mathbb R}}

\def\a{\alpha}

\def\E{{\mathbb E}}
\def\P{{\mathbb P}}

\begin{document}

\title{\bf laws of the iterated logarithm  for a class of iterated processes}


\author{Erkan Nane}
\address{Erkan Nane, Mathematics and Statistics, Auburn University, 221 Parker Hall,  Auburn, AL 36849}
\email{nane@auburn.edu}



\begin{abstract}Let $X=\{X(t), t\geq 0\}$ be a Brownian motion or a spectrally negative stable process of index $1<\a<2$.
Let $E=\{E(t),t\geq 0\}$ be the hitting time of a stable subordinator of index $0<\beta<1$ independent of $X$.
We use a connection between $X(E(t))$ and  the stable subordinator of
index $\beta/\a$ to derive information on the path behavior of $X(E_t)$. This is an extension of the connection of iterated Brownian motion and ($\frac{1}{4}$)-stable
subordinator due to
 Bertoin \cite{bertoin}.
 Using this connection, we obtain various laws of the iterated logarithm for $X(E(t))$. In particular, we establish law of the iterated logarithm for local time Brownian motion, $X(L(t))$, where $X$ is a Brownian motion (the case $\a=2$) and  $L(t)$ is the local time at zero of a stable process $Y$ of index $1<\gamma\leq 2$ independent of $X$. In this case $E(\rho t)=L(t)$ with $\beta=1-1/\gamma$ for some constant $\rho>0$.  This establishes the lower bound in the law of the iterated logarithm which we could not prove with the techniques of our paper  \cite{MNX}.  We also obtain exact small ball probability for $X(E_t)$ using ideas from \cite{aurzada}.
\end{abstract}


\maketitle

\newpage

\newpage
\section{Introduction}
Self-similar processes arise naturally in limit theorems of random walks and
other stochastic processes, and they have been applied to model various phenomena
in a wide range of scientific areas including telecommunications, turbulence,
image processing and finance. In this paper, we will study a class of self similar processes that are non-Gaussian and non-Markovian.

 To define the processes studied in this paper, let $X=\{X(t), t\geq 0\}$ be a Brownian motion or a spectrally negative stable L\'evy process of index $1<\a<2$ started at $0$. When $\a=2$, we denote $X=W$.
Let $E=\{E(t),t\geq 0\}$ be the hitting time of a stable subordinator of index $0<\beta<1$ independent of $X$.
Our aim in this paper is to obtain various laws of the iterated logarithm (LIL) for the process $Z=\{Z(t)=X(E(t)),\, t \ge 0\}$. This process is self-similar with index $H= \beta/\a$.

Many authors have constructed and investigated various classes of non-Gaussian
self-similar processes. See, for example, \cite{ST94} for information on self-similar
stable processes with stationary increments. Let $W'$ be an independent copy of $W$.  Define $B^1(t)=W(t), t\geq 0$ and $B^1(t)=W'(-t),\ t\leq 0$.  Burdzy \cite{burdzy1, burdzy2}
introduced the so-called \emph{iterated Brownian motion} (IBM), $A(t)=B^1(B(t))$, by replacing the time
parameter in $B^1(t)$ by an independent one-dimensional Brownian motion $B= \{B(t),
t \ge 0\}$. His work inspired  many researchers to explore the
connections between IBM (or other iterated processes) and PDEs,
 to establish potential theoretical
results
and to study its sample path properties, see \cite{allouba2,nane,nane2,nane3,nane5,nane6,xiao} and references therein.

The process $Z$ emerges as the scaling
limit of a continuous time random walk with heavy-tailed waiting times
between jumps \cite{coupleCTRW, limitCTRW}. Moreover,  $Z$ has a close connection to fractional partial
differential equations. Baeumer and Meerschaert \cite{fracCauchy}, and Meerschaert and Sheffler \cite{limitCTRW} showed
that the process $Z$ can be applied to provide a solution to the fractional
Cauchy problem. More precisely,
they proved that,
  $u(t,x)=\E_x [f(Z(t))]$ solves
the following fractional in time PDE
\begin{equation}\label{frac-derivative-0}
\frac{\partial^\beta}{\partial t^\beta}u(t,x) =
{G_x}u(t,x); \quad u(0,x) = f(x),
\end{equation}
where $G_x$ is the generator of the semigroup of $X_t$ and
$\partial^{\beta} g(t)/\partial t^\beta $ is
the Caputo fractional derivative in time, which can be defined as the inverse
Laplace transform of $s^{\beta}\tilde{g}(s)-s^{\beta -1}g(0)$, where
$\tilde{g}(s)=\int_{0}^{\infty}e^{-st}g(t)dt$ is the usual Laplace
transform.
Let $L(t)$ be the local time
at zero of a strictly stable process $Y(t)$ of index $1<\gamma\leq
2$ independent of $X$. The process, $W(L(t))$ is called  \emph{local time Brownian motion} in \cite{MNX}.
Recently Baeumer, Meerschaert and Nane \cite{bmn-07} further
established the
equivalence of the governing PDEs of $X(L(t))$ and $X(|B(t)|)$ when
 $\beta=1/2$. Here $B=\{B(t), t \ge 0\}$ is another Brownian
motion independent of $X$ and $Y$. When $\gamma=2$, the process $W(L(t))$ has also appeared in
the works of Borodin \cite{borodin1, borodin2}, Ikeda and Watanabe
\cite{ike-wata}, Kasahara \cite{kasahara}, and Papanicolaou et al.
\cite{papanicolaou}.

We are interested in sample path properties of $Z$. Since $Z$ is non-Gaussian,
non-Markovian and does not have stationary increments, the existing
theories on Markov and/or Gaussian processes can not be applied to $Z$
directly and some new tools need to be developed. In this regard, we mention that
Cs\'{a}ki,  F\"{o}ldes and R\'{e}v\'{e}sz   \cite{CCFR2} studied the Strassen
type LIL of $Z(t) = W(L(t))$ when $L(t)$ is the local
time at zero of a symmetric stable L\'{e}vy process; Meerschaert, Nane and Xiao
\cite{MNX} established a large deviation result, and uniform modulus
of continuity for $W(L(t))$. In particular they obtained upper bounds in  the LIL and  the modulus of continuity of $W(L(t))$. In this paper, we will establish  LIL results for $Z$.

This paper is a continuation of \cite{CCFR2} and \cite{MNX}. In section 2, we study small ball probability for $Z(t)$. Our results in Section 2 imply in particular that there can not be a Chung type LIL for $Z$ due to power decay in the small ball probability. In section 3, we study asymptotic
properties of $Z=\{Z(t),\, t \ge 0\}$ by first establishing a connection to a stable subordinator of index $\beta/\a$, which is an extension of the connection of iterated
Brownian motion and $\frac1 4$-stable subordinator due to Bertoin \cite{bertoin}
and then making use of limiting theorems on stable subordinators. \\

\section{Small ball Probability}
The study of small ball probability is useful  for studying Chung type LIL. In this section we show that there cannot be a Chung type LIL for $Z$.

In what follows  we will write $f(t) \approx g(t)$ as $t\to a$  to mean
that
for some positive $C_{1}$ and $C_{2}$, $C_{1}\leq \liminf_{t\to a}f(t)/g(t)\leq \limsup_{t\to a}f(t)/g(t) \leq C_{2}$.  We will also write $f(t) \sim
g(t)$, as
$t\to a $,  to mean $\lim_{t\to a}f(t) / g(t) \rightarrow 1$.


Next we introduce the processes that will be studied in this paper.
A L\'{e}vy process $X=\{X(t), t \ge 0\}$ with values in $\rr$ is called strictly
stable of index $\alpha \in (0, 2]$, $\a\neq 1$ if its characteristic function is given by
\begin{equation}\label{Eq:Chf}
\E\big[\exp (i\xi X(t))\big] = \exp\left(-t |\xi|^{\alpha} \frac{1+i \nu {\rm sgn}
(\xi)\tan (\frac{\pi\alpha}{2})}{\chi}\right),
\end{equation}
where $-1\leq  \nu \leq 1$ and $\chi >0$ are constants. In the terminology of
\cite[Definition 1.1.6]{ST94},
$-\nu$ and $\chi^{-1/\a}$ are respectively the skewness and scale parameters
of the stable random variable $X(1)$. When $\a=2$ and $\chi = 2$, $X$ is a standard Brownian
motion and, if $\chi = 1$, $X$ is a Brownian
motion running at twice the speed of standard Brownian motion and  we denote it by $W$. When $1<\a <2 $ and $\nu=1$, $X$ is called spectrally negative stable L\'evy process.

In general, many properties of stable L\'evy processes can be
characterized by the parameters $\alpha, \nu$ and $\chi$. For a systematic account
on L\'evy processes we refer to \cite{Bertoin96}.


Let  $E(t)$ be the continuous inverse of a
stable subordinator $D_t$ of index $0<\beta <1$:
$$
E(t)=\inf \{s:\ D(s)>t\}, \ (t\geq 0).
$$
$E(t)$ has also been called  the hitting time of stable subordinator of index $0<\beta<1$ in the literature, see, for  example, \cite{coupleCTRW}.

The next theorem is our first result on the small ball probability of $W(E(t))$.
\begin{theorem}\label{small-w}
Let $W(t)$ be a Brownian motion. Let  $E(t)$ be the continuous inverse of a
stable subordinator $D(t)$ of index $0<\beta <1$. Then
$$\lim_{u\to 0}u^{-2} \P\left(\sup_{0\leq t \leq 1}
|W(E(t))|\leq u \right)=\frac{32\Gamma (\beta )\sin(\beta \pi)}{\pi^4}\sum_{k=1}^{\infty}\frac{(-1)^{k-1}}{(2k-1)^{3}}.$$
\end{theorem}

\begin{proof}
From a well-known formula (see
Chung \cite{chung}):
$$
\P[\ \sup_{0\leq t \leq 1} |W(t)|\leq u \
]=\frac{4}{\pi}\sum_{k=1}^{\infty}\frac{(-1)^{k-1}}{2k-1}\exp
\left[ -\frac{(2k-1)^{2}\pi^{2}}{8u^{2}}\right],
$$

Since $E(t)$ is continuous and
nondecreasing, we use conditioning to get

\begin{eqnarray}
\P[\ \sup_{0\leq t\leq 1} |Z(t)|\leq u \ ]& =&
\P[\ \sup_{0\leq t\leq E(1)} |W(t)|\leq u  ] \nonumber\\
&=& \E\left[ \P\left( \sup_{0\leq t\leq 1} |W(t)|\leq
\frac{u}{  \sqrt{E(1)}}|E(1) \right) \right]\nonumber\\
&=  &\E\bigg(\frac{4}{\pi}\sum_{k=1}^{\infty}\frac{(-1)^{k-1}}{2k-1}\exp
\left[ -\frac{(2k-1)^{2}\pi^{2}E(1)}{8u^{2}}\right]\bigg)\label{infinite-small}
\end{eqnarray}

The Laplace transform of $E(t)$ ($E(t)$ inverse
stable subordinator of index $\beta$ ) (cf. \cite[Remark 2.8]{hop-loc})is
\begin{equation}\label{mittag-leffler}
\E[\exp(-sE(t))]=E^\beta(-st^\beta)=\sum_{n=0}^\infty \frac{(-st^\beta)^n}{\Gamma
(1+n\beta)}
\end{equation}
which is the Mittag-Leffler function evaluated at
$-st^\beta$.
And, by \cite[equation (13)]{krageloh}, for some $c>0$
\begin{equation}
\label{efie}
0\leq E^\beta(-a t^\beta)\leq
c/(1+a t^\beta).
\end{equation}
 Using Equation \eqref{efie} and the fact that Mittag-Leffler function is completely monotone  we obtain that the series in \eqref{infinite-small} is absolutely convergent, we get  by dominated convergence
 \begin{eqnarray}
 & &\E\bigg(\frac{4}{\pi}\sum_{k=1}^{\infty}\frac{(-1)^{k-1}}{2k-1}\exp
\left[ -\frac{(2k-1)^{2}\pi^{2}E(1)}{8u^{2}}\right]\bigg)\nonumber\\
&=& \frac{4}{\pi}\sum_{k=1}^{\infty}\frac{(-1)^{k-1}}{2k-1}\E\bigg(\exp
\left[ -\frac{(2k-1)^{2}\pi^{2}E(1)}{8u^{2}}\right]\bigg)\nonumber\\
&=&\frac{4}{\pi}\sum_{k=1}^{\infty}\frac{(-1)^{k-1}}{2k-1}E^\beta\bigg(
 -\frac{(2k-1)^{2}\pi^{2}}{8u^{2}}\bigg)\nonumber
\end{eqnarray}
 Therefore, using the fact that
 \begin{equation}\label{mittag-lim}
 E^\beta(-z)\sim cz^{-1}\ \mathrm{ as}\
z\to \infty
 \end{equation}
 with $c=\frac{\Gamma (\beta )\sin(\beta \pi)}{\pi}$, see, Kr\"{a}geloh \cite{krageloh} and taking the limit $u\to 0$ inside the sum using dominated convergence we obtain
\begin{eqnarray}
 \lim_{u\to 0}\frac{\P[\ \sup_{0\leq t\leq 1} |Z(t)|\leq u \ ]}{u^2}&=&\frac{4}{\pi}\sum_{k=1}^{\infty}\frac{(-1)^{k-1}}{2k-1}\lim_{u\to 0}\frac{E^\beta(-\frac{(2k-1)^{2}\pi^{2}}{8u^{2}})}{u^2}\nonumber\\
 &=&\frac{4}{\pi}\sum_{k=1}^{\infty}\frac{(-1)^{k-1}}{2k-1}\frac{8c}{(2k-1)^{2}\pi^{2}}\nonumber\\
 &=&\frac{32c}{\pi^3}\sum_{k=1}^{\infty}\frac{(-1)^{k-1}}{(2k-1)^{3}}
 \end{eqnarray}
\end{proof}
\begin{remark}
Since the exact small ball probability in Theorem \ref{small-w} is a power decay, we do not expect to get a Chung type LIL for $Z(t)$. We will obtain a Hirsch's integral test below  for $Z(t)$ using another method.
\end{remark}

Using the same conditioning method in Theorem \ref{small-w} and ideas in Aurzada and Lifshits \cite{aurzada} and Nane \cite{nane-ejp}, we also can get the following
\begin{theorem}
Let $U=\{U(t),t\geq 0\}$ be a self similar process of index $0<H<1$, and let $E(t)$ be the inverse of a stable subordinator of index $0<\beta<1$, independent of $U$. Let $\theta=1/H>0$ and suppose
\begin{equation}\label{log-small}
-\log\P\left(\sup_{0\leq t \leq 1}
|U(t)|\leq u \right)\sim k u ^{-\theta},\  \mathrm{as}\  u\to 0.
\end{equation}
Then
$$ \P\left(\sup_{0\leq t \leq 1}
|U(E(t))|\leq u \right)\approx u^{1/H}\  \mathrm{as}\  u\to 0.$$
\end{theorem}
\begin{remark}
In this theorem we get upper and lower bounds for the small ball probability due to the lack of an explicit formula for  the small ball probability for the process $U$.
\end{remark}

\begin{proof}
Let $\nu>0$. By assumption, for all $0<\epsilon<\epsilon_0=\epsilon_0(\nu)$,
$$
e^{-k(1+\nu)u^{-\theta}}\leq \P\left(\sup_{0\leq t \leq 1}
|U(t)|\leq u \right)\leq e^{-k(1-\nu)u^{-\theta}}
$$
This implies that there are constants $c_1, c_2>0$ that depend on $\epsilon_0$ such that for all $u>0$,
\begin{equation}\label{upper}
c_1e^{-k(1+\nu)u^{-\theta}}\leq \P\left(\sup_{0\leq t \leq 1}
|U(t)|\leq u \right)\leq c_2e^{-k(1-\nu)u^{-\theta}}
\end{equation}

Since $E_t$ is continuous and
nondecreasing, we use conditioning to get
\begin{eqnarray}
\P[\ \sup_{0\leq t\leq 1} |U(E(t))|\leq u \ ]& =&
\P[\ \sup_{0\leq t\leq E(1)} |U_{t}|\leq u,  ] \nonumber\\
&=& \E\left[ \P\left( \sup_{0\leq t\leq 1} |U(t)|\leq
\frac{u}{E(1)^H}|E(1) \right) \right]\nonumber\\
&\leq  &
c_2E\left[\exp\left(-k(1-\nu)u^{-\theta}E(1)\right)\right].\label{upper-upper}
\end{eqnarray}
we use the upper bound in \eqref{upper} to get the last inequality. We also obtain a lower bound for the small ball probability of $U(E)$ using the lower bound in \eqref{upper}.

Now taking limit as $u\to 0$ in \eqref{upper-upper} using equations \eqref{mittag-leffler}, \eqref{mittag-lim} and then letting $\nu\to 0$ we get the upper bound for the small ball probability of $U(E)$. With a similar argument we obtain the lower bound.
\end{proof}
Fractional
Brownian motion (fBm)  is a centered Gaussian process
$W^H=\{W^H(t), t \in \R\}$ with $W^H(0) = 0$ and covariance
function
\begin{equation}\label{Eq:fbmcov}
\E\big(W^H(s) W^H(t)\big) =\frac 1 2 \Big(|s|^{2H} + |t|^{2H} -
|s-t|^{2H}\Big),
\end{equation}
where $H \in (0, 1)$ is a constant.
$W^H$ is a self-similar process of index $H$. The small ball probability is given for $W^H$ with $\theta=1/H$ in \eqref{log-small}, see, for example, \cite{ST94}.
We obtain the small ball probability for local time fractional Brownian motion, $W^H(L(t))$, here $L(t)$ is the Local time at zero of an independent strictly stable process of index $1<\gamma\leq 2$, as a consequence of the following corollary using the fact that $E(\rho t)=L(t)$ with $\beta=1-1/\gamma$ for some constant $\rho>0$, see, for example, \cite{stone}.
\begin{corollary}
Let $W^H=\{W^H(t),t\geq 0\}$ be a fractional Brownian motion with index $0<H<1$, and let $E(t)$ be the inverse of a stable subordinator of index $0<\beta<1$, independent of $W^H$.
Then
$$ \P\left(\sup_{0\leq t \leq 1}
|W^H(E(t))|\leq u \right)\approx u^{1/H}\  \mathrm{as}\  u\to 0.$$
\end{corollary}
The process $W^H(E)$ emerges as the scaling
limit of a  correlated continuous time random walk with heavy-tailed waiting times
between jumps \cite{mnx1}.

Strictly stable process of index $0<\alpha\leq 2$ is a self-similar process of index $1/\alpha$ and the small ball probability is given with $\theta=\a$ in \eqref{log-small}, see, for example, \cite{bertoin}.
We obtain the small ball probability estimate of $Z$ as a consequence of
\begin{corollary}
Let  $X$ be a strictly stable process  of index $0<\alpha\leq 2$ and let $E_t$ be  the inverse of a stable subordinator of index $0<\beta<1$, independent of $X$.
 Then
$$ \P\left(\sup_{0\leq t \leq 1}
|X(E(t))|\leq u \right)\approx u^{\a}\  \mathrm{as}\  u\to 0.$$
\end{corollary}

\begin{remark}
Since the Laplace transform of the process $E(t)$ decays like $E(e^{-zE(1)})\sim c(\beta)/z$ as $z\to \infty$ for which the power of $z$ is independent of the index $\beta$,   the small ball probabilities does not depend on $\beta$ in the theorems in this section.
\end{remark}
For other extensions of small ball probabilities for iterated (fractional)
Brownian motions and iterated stable processes, see Aurzada and Lifshits \cite{aurzada}.


\section{Path regularity of $Z$}
We will consider the process $Z=\{Z(t),\, t \ge 0\}$ defined by
$$
Z(t)=X(E(t)), \quad \forall \ t \ge 0,
$$
where $X=W$ for $\a=2$ and $X$ is the spectrally negative stable L\'evy process with $1<\a<2$, $\nu=1$.

Let $S(t)=\sup_{0\leq s\leq t}X(s)$. Since $E(t)$ is  an increasing and
continuous process, we have
$$
\bar Z(t)=\sup_{0\leq s\leq t}Z(s)=S(E(t)).
$$
Let  $\sigma^\a$ be a stable subordinator of index $1/\a$, then
$S$ is the continuous inverse $\sigma^\a$, see Bingham \cite[Proposition 1]{bingham};.
$$
S(t)=\inf \{s:\ \sigma^\a(s)>t\}, \ (t\geq 0)
$$
 more precisely,
$$
\E\big(\exp(-s \sigma^\a(t))\big)=\exp(-t( s/c_1)^{1/\alpha})\ \ (s\geq 0, t\geq 0),
$$
where $c_1=\chi^{-1}\sec (\pi-\frac12\pi\a)$.

The inverse $D(t)$ of $E(t)$ is a stable subordinator of index $0<\beta <1$. Here $D$ has Laplace transform, $\E(e^{-sD(t)})=e^{-ts^\beta}$.
According to Bochner \cite{bochner}, $D \circ \sigma^\a$ is a stable subordinator
with index $\frac\beta \a$, and more precisely,

$$
\log \big(\E[e^{-sD\circ \sigma^\a(t)}] \big)= mt\int_0^\infty (e^{-sx}-1)x^{-(1+\beta/\a)}dx
 =-c_1^{-1/\a}ts^{\beta/\a}
$$
by integration by parts, we can show that $m=c_1^{-1/\a}(\beta/\a)[\Gamma(1-\beta/\a)]^{-1}$.
Define
$$
\mu=(\Gamma(1-\beta/\a)m)^{\a /(\a-\beta)}(\a-\beta)/\beta.
$$
Let $\lambda=(\beta/\a)/(1-\beta/\a)=\beta/(\a-\beta)$.
Plainly, $D \circ \sigma^\a$ is the right-continuous inverse of $S\circ E$.

\begin{prop}
The right-continuous inverse of the supremum of $Z(t)$ is
$$
\inf \{t: \ \bar Z(t) >\cdot\} =\sigma^{Z}(\cdot),
$$
where $\sigma^{Z}$ is a stable subordinator with Laplace exponent $s\to c_1^{-1/\a}\, s^{\beta/\a}$,
\end{prop}

We can now apply certain results in the literature on stable subordinators to
derive asymptotic results about $Z$.


Our first result on LIL is the following
\begin{theorem}\label{k-lil}
Let $X$ be a spectrally negative stable L\'evy process  with $1<\a\leq 2$, $\nu=1$, and let $E(t)$ be the inverse to a stable subordinator of index $0<\beta<1$ independent of $X$. Then
$$
\limsup \frac{Z(t)}{t^{\beta/\a}(\log|\log t|)^{(\a-\beta)/\a}}=(c(\beta,\a))^{\beta/\a} \ \ \mathrm{a.s.}
$$
both as  $t\to 0^+$ and $t\to\infty$. Here $c(\beta,\a)=\mu^{1/\lambda}$

\end{theorem}

When $\a=2$, the upper bound in this large time  LIL result was proved by
Meerschaert, et al. \cite{MNX} for the case $E_{\rho t}=L_t$ is the local time at zero  of a strictly
stable process of index $1<\gamma\leq 2$ and in this case $\beta=1-1/\gamma$. Here $\rho>0$ given by
$$
\rho^{-(1-1/\gamma)}=\pi^{-1}\Gamma (1+1/\gamma)\Gamma(1-1/\gamma)\chi^{-1}Re \{[1+i\nu\tan(\pi\gamma/2)]^{-1/\gamma}\},
$$
see Stone \cite{stone}. Hence we obtain the following LIL for local time Brownian motion $W(L(t))$.
\begin{corollary}\label{local-time-lil}
Let $W$ be a Brownian motion, and  let $L(t)$ be the local time at zero of a stable process of index $1<\gamma\leq 2$, independent of  $W$. Then
$$
\limsup \frac{W(L(t))}{t^{(\gamma-1)/2\gamma}(\log|\log t|)^{(\gamma+1)/2\gamma}}=(\rho c(\beta=1-1/\gamma,2))^{(\gamma-1)/2\gamma} \ \ \mathrm{a.s.}
$$
both as  $t\to 0^+$ and $t\to\infty$.
\end{corollary}

\begin{remark}
Corollary \ref{local-time-lil} is an  extension of the results in \cite{MNX,CCFR2}.
Also the case $\gamma=2$ gives the LIL for  a version of iterated Brownian motion, $W(|B(t)|)$, which was obtained in \cite{burdzy1, ccfr1}, here $B$ is a Browninan motion independent of $W$.
\end{remark}

\begin{proof}[Proof of Theorem \ref{k-lil}]
Applying the LIL for stable subordinators
proved by  Fristedt \cite{fristedt} (see also Breiman \cite{breiman})
to $D\circ \sigma^\a$, we have
$$
\liminf \frac{D\circ \sigma^\a(t)}{t^{\a/\beta}(\log |\log t|)^{(\beta-\a)/\beta}}
=\mu^{1/\lambda}=c(\beta,\a) \ \ \mathrm{a.s.}
$$
both as  $t\to 0^+$ and $t\to\infty$.

Since $S^\a\circ E(D\circ \sigma^\a_t)=t$, we deduce that the probability that
$$
t\leq S^\a\circ E (ct^{\a/\beta}(\log |\log t|)^{(\beta-\a)/\beta})
$$
infinitely often as $t\to 0^+$ (or as $t\to \infty$) equals $1$ for
$c>c(\beta,\a)$ and $0$ for $c<c(\beta,\a)$. Recall that
$t\to t^{\beta/\a}(\log|\log t|)^{(\a-\beta)/\a}$ is an asymptotic
inverse of $t\to t^{\a/\beta}(\log |\log t|)^{(\beta-\a)/\beta} $, so
$$
\limsup \frac{S^\a\circ E(t)}{t^{\beta/\a}(\log|\log t|)^{(\a-\beta)/\a}}=(c(\beta,\a))^{\beta/\a} \ \ \mathrm{a.s.}
$$
both as  $t\to 0^+$ and $t\to\infty$.

One can replace $S \circ E(t)$  with $Z(t)$.
\end{proof}

We next obtain the Kolmogorov's integral test for $Z$.

\begin{theorem}
Let $X$ be a spectrally negative stable L\'evy process with $1<\a\leq 2$, $\nu=1$, and let $E(t)$ be the inverse to a stable subordinator of index $0<\beta<1$ independent of $X$. Then
$$
\P[Z(t)\geq t^{\beta/\a}f(t)\ \mathrm{infinitely \ often \ as }\ t\to\infty ]
=0\ \mathrm{ or } \ 1,
$$
according as
$$
\int_1^\infty f(t)^{1/(\a-\beta)} \exp(-\mu [f(t)]^{\a/(\a-\beta)})\frac{dt}{t}
 $$
converges or diverges. A similar result holds for small times by replacing the integral over $(1,\infty)$ with the integral over $(0,1)$.
\end{theorem}
\begin{proof}
Breiman \cite{breiman} obtained a refinement
of Fristedt's LIL for stable subordinators. It states that if $\varphi:(0,\infty)\to (0,\infty) $  is a decreasing function, then
\begin{equation}\label{kolmogorov}
\P[D\circ \sigma(t)\leq t^{\a/\beta}\varphi(t) \ \mathrm{infinitely \ often \ as }\  t\to\infty]
=0\ \mathrm{or }\ 1
\end{equation}
according as the integral
$$
I(\varphi)=\int_1^\infty [\varphi(t)]^{-\lambda/2}
e^{-\mu[\varphi(t)]^{-\lambda}}\, \frac{dt}{t}
$$
converges or diverges. Here $\lambda=\beta/(\a-\beta)$.

Let $f:(0,\infty)\to (0,\infty) $ be an increasing function, then by first replacing  $t$
by $t^{\beta/\a}f(t)$ in \eqref{kolmogorov}, we get that
$$D\circ \sigma^\a({t^{\beta/\a}f(t)})\leq tf(t)^{\a/\beta}\varphi(t^{\beta/\a}f(t))
\ \mathrm{infinitely \ often \ as }\  t\to\infty
$$
and then taking $S\circ E$ of both sides we get the probability that
$$
t^{\beta/\a}f(t)\leq S \circ E(tf(t)^{\a/\beta}\varphi(t^{\beta/\a}f(t)))
 t\to\infty
$$
equals $0$ or $1$ according as $I(\varphi)<\infty$ or $I(\varphi)=\infty$.
If we now choose $\varphi$ such that
$\varphi(t^{\beta/\a}f(t))=f(t)^{-\a/\beta}$, we then see that
$$
\P[S\circ E(t)\geq t^{\beta/\a}f(t)\ \mathrm{infinitely \ often \ as }\ t\to\infty ]
=0\ \mathrm{ or } \ 1,
$$
according as  $I(\varphi)<\infty$ or $I(\varphi)=\infty$. An easy calculation
shows that the latter alternative reduces to decide whether the integral
(replace $\varphi(t)$ with $f(t)^{-\a/\beta}$)
$$
\int_1^\infty f(t)^{1/(\a-\beta)} \exp(-\mu [f(t)]^{\a/(\a-\beta)})\frac{dt}{t}
 $$
converges or diverges. We can replace $S\circ E$ by $X\circ E=Z$, which gives
 the Kolmogorov's integral test for the upper functions of $Z$ for large times. A similar
 result holds for small times by replacing the integral over $(1,\infty)$ with the integral over $(0,1)$.

\end{proof}
We next obtain Hirsch's integral test for $Z(t)$

\begin{theorem}
Let $X$ be a spectrally negative stable L\'evy process $1<\a\leq 2$, $\nu=1$, and let $E(t)$ be the inverse to a stable subordinator of index $0<\beta<1$ independent of $X$. Then
$$
\liminf_{t\to\infty}\bar Z(t)/\varphi(t)=0 \ \mathrm{or }\ \infty
$$
according as the integral $\int_1^\infty \varphi(t) t^{-(1+\beta/\a)}dt$
diverges or converges. A similar result holds for small times by replacing the integral over $(1,\infty)$ with the integral over $(0,1)$.
\end{theorem}
\begin{proof}
 Recall that if $f:(0,\infty)\to (0,\infty)$ is an increasing function, then
$$
\limsup_{t\to\infty}D\circ \sigma^\a(t)/f(t)=0\ \mathrm{or }\ \infty\ \ \mathrm{a.s.}
$$
according as $\int_1^\infty f(t)^{-\beta/\a}dt$ converges or diverges. See Theorem
11.2 on page 361 in Fristedt \cite{fristedt-2}. Consider first the case Limsup is equal to $0$.
Let $\epsilon>0$. Then,
$$
D\circ \sigma^\a(t)\leq \epsilon f(t) \ \ \mathrm{for\  all}\ t \ \mathrm{large\ \ a.s}.
$$
Now taking $S\circ E$ of both sides we get
$$
t\leq S\circ E(\epsilon f(t)) \ \ \mathrm{for\  all}\ t \ \mathrm{large\ \ a.s}.
$$

Using scaling of $S\circ E$ we get
$$
t\leq (\epsilon f(t))^{\beta/\a}t^{-\beta/\a}S\circ E(t) \ \ \mathrm{for\  all}\
t \ \mathrm{large\ \ a.s}.
$$
hence
$$
t^{1+\beta/\a}(\epsilon f(t))^{-\beta/\a}\leq S\circ E(t) \ \ \mathrm{for\  all}\ t \
\mathrm{large\ \ a.s}.
$$
We use  a similar argument for the case Limsup is infinite.
Hence, setting $\varphi(t)=t^{1+\beta/\a}f(t)^{-\beta/\a}$ we get
$$
\liminf_{t\to\infty}S\circ E(t)/\varphi(t)=0 \ \mathrm{or }\ \infty
$$
according as the integral $\int_1^\infty \varphi(t) t^{-(1+\beta/\a)}dt$
diverges or converges. Again, there is a similar result for small times.
\end{proof}

We next obtain a modulus of continuity result for the supremum process $\bar Z$.

\begin{theorem}Let $X$ be a spectrally negative stable L\'evy process with $1<\a\leq 2$, $\nu=1$, and let $E(t)$ be the inverse to a stable subordinator of index $0<\beta<1$ independent of $X$. Then
$$
\lim_{\epsilon\to 0^+}\sup_{0\leq t\leq 1, 0<h<\epsilon}
\frac{\bar Z({t+h})-\bar Z({t})}{h^{\beta/\a}[\log (1/h)]^{(\a-\beta)/\a}}=d(\alpha,\beta)=\left( \a c_2(\beta/\a)/\beta\right)^{(1-\beta/\a)} \ \mathrm{a.s.}
$$
where $c_2=c_2(\beta/\a)=(1-\beta/\a)(\beta/\a)^{\beta/(\a-\beta)}$.
\end{theorem}

\begin{proof}
Since $\sigma^{Z}$ is a stable subordinator
of index $\beta/\a$, Theorem 1 in Hawkes \cite{hawkes} gives (here we change the L\' evy measure of $\sigma^\a$ so that the Laplace exponent is $s\to s^{\beta/\alpha}$)
$$
\lim_{\epsilon\to 0^+}\inf_{0\leq t\leq 1, 0<h<\epsilon}\frac{\sigma^{Z}({t+h})-\sigma^{Z}({t})}{h^{\a/\beta}[\log (1/h)]^{-(\a-\beta)/\beta}}=c_2^{(\a-\beta)/\beta} \ \mathrm{a.s.}
$$
where $c_2=c_2(\beta/\a)=(1-\beta/\a)(\beta/\a)^{\beta/(\a-\beta)}$. Taking
inverses as in Hawkes \cite{hawkes} we obtain
$$
\lim_{\epsilon\to 0^+}\sup_{0\leq t\leq 1, 0<h<\epsilon}
\frac{\bar Z({t+h})-\bar Z({t})}{h^{\beta/\a}[\log (1/h)]^{(\a-\beta)/\a}}=d(\alpha,\beta)=\left( \a c_2(\beta/\a)/\beta\right)^{(1-\beta/\a)} \ \mathrm{a.s.}
$$

\end{proof}

\begin{remark}
We can also  deduce the analog of the Chung-Erd\"os-Sirao integral test
for $\bar Z$ from Theorem 7.2.1 in Mijnheer \cite{mijnheer}. We refer to
Meerschaert, Nane and Xiao \cite{MNX} for the modulus of continuity of $W(L(t))$, where $L(t)$ is the local time at zero of an independent strictly stable process $Y$ of index $1<\gamma\leq 2 $. We also refer to Khoshnevisan and Lewis \cite{khosh-lew} for the modulus of continuity of IBM $A(t)$.
\end{remark}


\begin{remark}
Bertoin \cite{bertoin}, remarked that similar results holds for $n$-iterated Brownian motions. In this regard we
get similar LIL and integral tests for $n$-iterated $E(t)$ processes.
We do not know whether this is the case for $n$-iterated spectrally negative L\'evy processes.
\end{remark}
\noindent{\bf Acknowledgment} Author would like to thank Yimin Xiao for discussion of the results in this paper.



\end{document}